\documentclass[a4paper,10pt]{amsproc}
\usepackage{graphicx}
\usepackage{amsfonts, amsmath, amssymb, graphicx, mathrsfs}
\usepackage[all]{xy}

\newdir{ >}{!/6pt/@{ }*:(1,0.2)@_{>}*:(1,-0.2)@^{>}}

\theoremstyle{plain}

\newtheorem{theorem}{Theorem}[section]
\newtheorem{lemma}[theorem]{Lemma}

\theoremstyle{definition}

\newtheorem{counter example}[theorem]{Counter Example}

\newtheorem{corollary}[theorem]{Corollary}

\numberwithin{equation}{section}

\begin{document}
\Large{
\title[$z^\circ$-ideals in intermediate rings]{$z^\circ$-ideals in intermediate rings of ordered field valued continuous functions}
\author[S. Bag]{Sagarmoy Bag}
\address{Department of Pure Mathematics, University of Calcutta, 35, Ballygunge Circular Road, Kolkata 700019, West Bengal, India}
\email{sagarmoy.bag01@gmail.com}

%Information for second author (if needed):
\author[S.K.Acharyya]{Sudip Kumar Acharyya}
\address{Department of Pure Mathematics, University of Calcutta, 35, Ballygunge Circular Road, Kolkata 700019, West Bengal, India}
\email{sdpacharyya@gmail.com}
\thanks{The first author thanks the NBHM, Mumbai-400 001, India, for financial support}

%Information for third author:
\author[D. Mandal]{Dhananjoy Mandal}
\address{Department of Pure Mathematics, University of Calcutta, 35, Ballygunge Circular Road, Kolkata 700019, West Bengal, India}
\email{dmandal.cu@gmail.com}

%                                                                                                                           %
%         Please use the current 2010 Mathematics Subject Classification:             %
%         http://www.ams.org/mathscinet/msc/                                                        %
%         http://www.zentralblatt-math.org/msc/en/                                                 %
%%%%%%%%%%%%%%%%%%%%%%%%%%%%%%%%%%%%%%%%%%%%%%%%%%%

\thanks {}

\maketitle
\section{Abstract}
A proper ideal $I$ in a commutative ring with unity is called a $z^\circ$-ideal if for each $a$ in $I$, the intersection of all minimal prime ideals in $R$ which contain $a$ is contained in $I$. For any totally ordered field $F$ and a completely $F$-regular topological space $X$, let $C(X,F)$ be the ring of all $F$-valued continuous functions on $X$ and $B(X,F)$ the aggregate of all those functions which are bounded over $X$. An explicit formula for all the $z^\circ$-ideals in $A(X,F)$ in terms of ideals of closed sets in $X$ is given. It turns out that an intermediate ring $A(X,F)\neq C(X,F)$ is never regular in the sense of Von-Neumann. This property further characterizes $C(X,F)$ amongst the intermediate rings within the class of $P_F$-spaces $X$. It is also realized that $X$ is an almost $P_F$-space if and only if each maximal ideal in $C(X,F)$ is $z^\circ$-ideal. Incidentally this property also characterizes $C(X,F)$ amongst the intermediate rings within the family of almost $P_F$-spaces. 

\section{introduction}

Let $F$ be a totally ordered field equipped with its order topology. For any topological space $X$ the set $C(X,F)$ of all $F$-valued continuous functions on $X$ constitutes a commutative lattice ordered ring with identity if the relevant operations are defined pointwise on $X$. The subset $B(X,F)$ of $C(X,F)$ containing all those functions $f$, which are bounded over $X$ in the sense that $\lvert f\rvert \leq \lambda $ for some $\lambda >0$ in $F$, makes a subring as well as a sublattice of $C(X,F)$ with the special choice $F=\mathbb{R}$, these two rings $C(X,F)$ and $B(X,F)$ reduce to the well known rings $C(X)$ of all real valued continuous functions on $X$ and its subring $C^*(X)$ of all bounded real valued continuous functions on $X$ respectively. There is a nice interplay existing between the topological structure of $X$ and the algebraic structure of $C(X)$ and $C^*(X)$ both. Problem of this kind are beautifully addressed in the classic monograph \cite{GJ}. We would like to mention in this context that a good many results related to this interplay are still valid in a wider setting if $C(X)$ is replaced by $C(X,F)$ and $C^*(X)$ by $B(X,F)$ for any totally ordered field $F$ and this is best realized if we stick to completely $F$-regular spaces $X$ is called completely $F$-regular if it is Hausdorff and given a point $x$ in $X$ and a closed set $K$ in $X$ missing the point $x$, there exists an $f$ in $B(X,F)$ such that $f(x)=0$ and $f(K)=1$. Thus complete $F$-regularity with the choice $F=\mathbb{R}$ reduces to complete regularity. However if $F\neq \mathbb{R}$, then $X$ becomes completely $F$-regular if and only if it is zero-dimensional. So zero-dimensionality of a Hausdorff space $X$ can be realized as a kind of separation axiom effected by $F$-valued continuous functions on $X$ with $F\neq \mathbb{R}$. Problems of this kind are already investigated in \cite{ACG2004}, \cite{AKPG2004} and \cite{BBNW1975}. A ring $A(X,F)$ that lies between the two rings $B(X,F)$ and $C(X,F)$ is called an intermediate ring in \cite{AR2016}. There is a particularly interesting common property shared by all these intermediate rings viz that the structure spaces of all such rings are one and the same. In \cite{AR2016}, this common structure space is realised as $\beta _F X$, the set of all $z_F$-ultrafilters on $X$ equipped with the Stone-topology. It may be mentioned that the structure space of a commutative ring $R$ with 1 stands for the set of all maximal ideals of $R$, endowed with the well known hull-kernel topology [ Chapter 7, exercise 7M, \cite{GJ}. Incidentally a $z_F$-ultrafilter on $X$ is a family of zero sets of $F$-valued continuous functions on $X$, which is maximal with respect to having finite intersection property. In the present article, our main intention is to highlight a few dissimilarities between the parent ring $C(X,F)$ and its proper intermediate subrings $A(X,F)$'s. Such a programme has already been initiated in the papers \cite{MSW2017}, \cite{SW2014}, \cite{SW2015}. In these papers the authors have characterized $C(X)$ amongst the intermediate rings lying between $C^*(X)$ and $C(X)$ within the categories of some well chosen class of spaces. Truely speaking, the present paper is in tune with this programme. To accomplish this, we need two particular types of ideals viz $z^\circ$-ideals and $z$-ideals in the intermediate rings $A(X,F)$'s mentioned above. An ideal unmodified in a ring, in this paper will always stand for a proper ideal; unless we specify the contrary.

Let $R$ be a commutative ring with unity. For each $'a'$ in $R$, let $P_a(M_a)$ stand for the intersection of all minimal prime ideals (maximal ideals) of $R$, which contain $a$. An ideal $I$ in $R$ is called a $z^\circ$-ideal (respectively $z$-ideal ) if for each $a$ in $I$, $P_a\subseteq I$ ( respectively $M_a\subseteq I$). It was proved in \cite{M1989} that a $z^\circ$-ideal of $R$ is a $z$-ideal if and only if the Jacobson radical of $R$ is zero. In particular in our situation each $z^\circ$-ideal in any intermediate ring $A(X,F)$ is a $z$-ideal. $R$ is called a reduced ring if 0 is the only nilpotent element of $R$. It is clear that each intermediate ring $A(X,F)$ is reduced. The following useful description of minimal prime ideals in $R$ due to kist \cite{K1963} is recorded in [\cite{HJ1965}, lemma 1.1]:

\begin{theorem}
A prime ideal $P$ in a commutative ring $R$ is a minimal prime ideal if and only if given $a\in P$, there is $b\in R\setminus P$ such that $ab$ is a nilpotent member of $R$( in particular $ab=0 $ if $R$ is reduced ).
\end{theorem} 

It follows from this theorem that each non zero element of a minimal prime ideal in a reduced ring $R$ is a divisor of zero. Consequently in such a ring $R$ all proper $z^\circ$-ideals contain divisors of zero of $R$ only. Incidentally it was recorded in \cite{AKA1999}, Example 3 that for a non compact space $X$, the ideal $C_K(X)$ of all functions in $C(X)$ with compact support is a $z^\circ$-ideal. In the present article, we have shown more generally that whenever $\mathcal{P}$ is an ideal of closed sets in $X$ in the following sense that if $A, B \in \mathcal{P}$ and $C$ is a closed in $X$ contained in $A$, then $A\cup B\in \mathcal{P}$ and $C\in \mathcal{P}$, it follows that $C_\mathcal{P}(X,F)=\{ f\in C(X,F): cl_X(X\setminus Z(f))\textit{ lies on } \mathcal{P}\}$ is a $z^\circ$-ideal of $C(X,F)$ and more generally for an intermediate ring $A(X,F), C_{\mathcal{P}}(X,F)\cap A(X,F)$ is a $z^\circ$-ideal of $A(X,F)$, here $Z(f)=\{ x\in X: f(x)=0\}$ is the zero set of $f$ in $C(X,F)$. It is interesting to note as we have observed that the converse of this statement is also true. This means that if $I$ is a proper $z^\circ$-ideal of $A(X,F)$, then there exists an ideal $\mathcal{P}[I]$ of closed sets in $X$ for which $C_{\mathcal{P}[I]}(X,F)\cap A(X,F)=I$ [Theorem \ref{4} of the present article ]. Thus an explicit description of $z^\circ$-ideals of intermediate rings $A(X,F)$ can be given via ideals of closed sets in $X$. On using this last fact, we have established that if $I$ is a proper ideal of $A(X,F)$, containing divisors of zero only, then $I$ can be extended to a $z^\circ$-ideal of $A(X,F)$ [ Theorem \ref{35} of the present article ]. This extends a result proved in \cite{AKA1999}, Theorem 2.5. It is well known that the ring $C(X)$ can be regular in the sense of Von-Neuman for a well designated class of spaces $X$ viz the class of $P$ spaces initiated in \cite{GJ}. More generally given a totally ordered field $F$, a completely $F$-regular space $X$ is called a $P_F$ space if the ring $C(X,F)$ is regular in Von-Neuman sense. For $F=\mathbb{R}$, $P_F$-spaces reduce to  $P$-spaces. Nevertheless if $F$ is a Cauchy complete totally ordered field with co-finality index equal to $\omega_0$, then $P_F$-spaces are nothing but $P$-spaces. Indeed there do exist plethora of non archimedean Cauchy complete totally ordered fields $F$ with co-finality index $\omega_0$. For more information about these $P_F$-spaces see \cite{AKPG2004}, one can ask a pertinent question: Can an intermediate ring $A(X,F)$, properly contained in $C(X,F)$ be ever regular? We have given a negative answer to this question by using some properties of $z^\circ$-ideals [Theorem \ref{6} of the present article ]. A special case of this result on choosing $F=\mathbb{R}$ reads: a ring $A(X)$ containing $C^*(X)$ and contained properly in $C(X)$ is never regular, which was proved recently in \cite{AB2013} and \cite{MSW2017}. Incidentally we have offered a new characterization of $P_F$-spaces in the following manner: a completely $F$-regular space $X$ is a $P_F$-space if and only if every ideal of $C(X,F)$ is a $z^\circ$-ideal ( Theorem \ref{8}). This property therefore characterizes in the class of $P_F$ spaces, the ring $C(X,F)$ amongst all the intermediate rings of $F$-valued continuous functions. As recorded in \cite{L1977}, $P$-spaces are fairly rare. A larger class of spaces, the almost $P$-spaces consisting precisely of those spaces in which non-empty zero sets have non-empty interior. ( equivalently every non-empty $G_{\delta}$ set has non-empty interior ) have been introduced in \cite{L1977}. Almost $P$-spaces are far more abundant than $P$-spaces ( see \cite{L1977}, for example of almost $P$ spaces ). These spaces have been characterized via $z^\circ$-ideals and $z$-ideals in \cite{AKA1999}, Theorem 2.14, in the following manner: A space $X$(Tychonoff) is almost $P$-space if and only if each $z$-ideal of $C(X)$ is a $z^\circ$-ideal. We have made a generalization of this fact by initiating corresponding to a totally ordered field $F$, an almost $P_F$-space $X$ as follows: a completely $F$-regular space $X$ is called an almost $P_F$-space if each non-empty zero set in $X$ of $F$-valued continuous functions has non-empty interior. It is easy to check by using the complete $F$-regularity of $X$ that $X$ is almost $P_F$-space if and only if each zero set in $X$ is regular closed i.e. for any $f\in C(X,F), Z(f)= cl_X (int_X Z(f))$. It is trivial that an almost $P_F$ space with $F=\mathbb{R}$ is just an almost $P$-space. We are not aware of any totally ordered field $F$ and a suitable completely $F$-regular space $X$ such that $X$ is almost $P_F$ space with out being an almost $P$-space. On the contrary, there exist pretty many totally ordered fields $F$, for which the class of almost $P_F$-spaces and almost $P$ spaces do coincide. We have offered a proof of this assertion [Theorem \ref{9}] in the present article. It is established in this paper [ Theorem \ref{12}] that a completely $F$-regular space $X$ is almost $P_F$ if and only if all the fixed maximal ideals of any prescribed intermediate ring $A(X,F)$ are $z^\circ$-ideals. This leads to a further characterization of almost $P_F$ spaces, [Theorem \ref{13}], which says that almost $P_F$ space $X$ are precisely those for which each maximal ideal in $C(X,F)$ is a $z^\circ$-ideal. Incidentally, this property characterizes $C(X,F)$ amongst all the intermediate rings of $F$-valued continuous functions on $X$ (Theorem \ref{14}). We conclude this introductory section by adopting the following notational convention: $Z(X,F)$ will stand for the family of all zero sets of functions in $C(X,F)$ and $A(X,F)$ will designate for a typical intermediate subring of $C(X,F)$. Also a space $X$ will stand for a completely $F$-regular space.

\section{$z^\circ$-ideal in intermediate rings}

We start with the following convenient formula for the basic $z^\circ$-ideal $P_a$ for an element $a$ of a reduced ring $R$ as recorded in \cite{AKA2007}, ( Proposition 1.5).

\begin{theorem}\label{1}
$P_a=\{ b\in R: Ann(a)\subseteq Ann(b)\}$, here $Ann(a)\equiv \{ c\in R: ac=0\}$ is the annihilator of $a$ in $R$.
\end{theorem}

By using the technique of proof of Lemma 2.1 in \cite{AKA1999} using Theorem \ref{1} and taking note of the fact that each bounded function in  $C(X,F)$ belongs to an intermediate ring $A(X,F)$, we get the following proposition:

\begin{theorem}\label{2}
For a space $X$ and for $f, g\in A(X,F)$, $Ann(f)\subseteq Ann(g)$ if and only if $int_X Z(f)\subseteq int_X Z(g)$. Here $Ann(f)$ stands for the annihilator of $f$ in $A(X,F)$.
\end{theorem}

We write the following formula for the basic $z^\circ$-ideals in $A(X,F)$, which is obtained by combining theorem \ref{1} and Theorem \ref{2} and which we will need to use several times in the present paper.

\begin{theorem}\label{3}
For $f\in A(X,F), P_f=\{ g\in A(X,F): int_X Z(f)\subseteq int_X Z(g)\}$.
\end{theorem}

On using this theorem, we obtain the following new topological characherization of $z^\circ$-ideals in the intermediate rings.

\begin{theorem}\label{4}
For any ideal $\mathcal{P}$ of closed sets in $X$, $C_{\mathcal{P}}(X,F)\cap A(X,F)\equiv \{f\in A(X,F): cl_X(X\setminus Z(f))\in \mathcal{P}\}$ is a $z^\circ$-ideal of $A(X,F)$.

Conversely, if $I$ is a $z^\circ$-ideal of $A(X,F)$, then there exists an ideal $\mathcal{P}[I]$ of closed sets in $X$ such that $I= C_{\mathcal{P}[I]}(X,F)\cap A(X,F)$.
\end{theorem}

\begin{proof}
Let $f\in C_{\mathcal{P}}(X,F)\cap A(X,F)$. We shall show that $P_f(\equiv $ the intersection of all minimal prime ideals of $A(X,F)$ which contain $f)\subseteq C_\mathcal{P}(X,F)$. Choose $g\in P_f$. Then by Theorem \ref{3}, $int_X Z(f)\subseteq int_X Z(g)$. This implies that $cl_X(X\setminus Z(g))\subseteq cl_X (X\setminus Z(f))$. Since $f\in C_\mathcal{P}(X,F)$ and $\mathcal{P}$ is an ideal of closed sets in $X$, it follows that $cl_X(X\setminus Z(g))\in \mathcal{P}$ i.e. $g\in C_\mathcal{P}(X,F)$.

Conversely, let $I$ be a $z^\circ$-ideal of $A(X,F)$. Then the family $T=\{ cl_X(X\setminus Z(f)): f\in I\}$ is closed under finite union. Suppose $\mathcal{P}[I]$ is the ideal of closed sets in $X$, generated by this family i.e. $\mathcal{P}[I]=\{ K: K \textit{ is closed in } X \textit{ and }K\subseteq \textit{ some member of } T\}$. We assert that $I=C_{\mathcal{P}[I]}(X,F)\cap A(X,F)$. It is plain that $I\subseteq C_{\mathcal{P}[I]}(X,F)\cap A(X,F)$. To prove the other containment let $g\in C_{\mathcal{P}[I]}(X,F)\cap A(X,F)$. Then  there exists $f\in I$ such that $cl_X(X\setminus Z(g))\subseteq cl_X(X\setminus Z(f))$. It follows that $int_X Z(g)\supseteq int_X Z(f)$ i.e. $g\in P_f$ in the ring $A(X,F)$. As $f\in I$ and $I$ is a $z^\circ$-ideal, this implies that $g\in I$. Thus $C_{\mathcal{P}[I]}(X,F)\cap A(X,F)\subseteq I$.
\end{proof}

For more information on ideal of closed sets in $X$, see \cite{AG2010} and \cite{AG2012}.
The following proposition is a consequence of Theorem \ref{4} and gives an alternative approach to prove a similar kind of theorem in [\cite{AKA1999} Theorem 2.5].

\begin{theorem}\label{35}
An ideal $I$ in $A(X,F)$ containing divisors of zero only can be extended to a proper $z^\circ$-ideal in this ring.
\end{theorem}

\begin{proof}
It is easy to check that an $f\in C(X,F)$ is a divisor of zero in the ring $C(X,F)$ if and only if $int_X Z(f)\neq \phi$. Since an $f\in A(X,F)$ is a divisor of zero in $A(X,F)$ if and only if it is a divisor of zero in $C(X,F)$, as can be immediately verified, it follows that an $f\in A(X,F)$ is a divisor of zero in $A(X,F)$ if and only if $int_X Z(f)\neq \phi$. Therefore for each $f\in I, f\neq 0, int_X Z(f)$ is a non-empty proper open subset of $X$. Consequently $cl_X (X\setminus Z(f))=X\setminus int_X Z(f)$ is a non-empty proper closed subset of $X$. Accordingly $\{ cl_X (X\setminus Z(f)): f\in I\}$ is a family of closed sets in $X$, which is closed under finite unions and which excludes the set $X$. Let $\mathcal{P}[I]$ be the ideal of closed sets in $X$, generated by the last family. Then it is clear that $I\subseteq C_{\mathcal{P}[I]}(X,F)\cap A(X,F)$. We further note that, since $X\notin \mathcal{P}[I]$, this implies on using theorem \ref{4} that $C_{\mathcal{P}[I]}(X,F)\cap A(X,F)$ is a proper $z^\circ$-ideal of $A(X,F)$.
\end{proof}

Remark: A special case of this theorem on choosing $F=\mathbb{R}$ and $A(X,F)=C(X)$ was establish in [\cite{AKA1999} Theorem 2.5] where the authors have offered an existential proof of it by using the principle of transfinite induction. We would like to mention that the proof of the theorem \ref{35} above is rather constructive.

To show that an intermediate ring, properly contained in $C(X,F)$ is never regular, we require the following auxiliary result.

\begin{lemma}\label{5}
 An intermediate ring $A(X,F)$ is an absolutely convex subring of $C(X,F)$ in  the sense that if $\lvert f\rvert \leq \lvert g\rvert, f\in C(X,F)$ and $g\in A(X,F)$, then $f\in A(X,F)$. 
\end{lemma}

\begin{proof}
As $A(X,F)$ is a lattice ordered ring, easily verifiable, and $\frac{f}{1+\lvert g\rvert}\in B(X,F)$, it follows that $f=\frac{f}{1+\lvert g\rvert}(1+\lvert g\rvert)$ belongs to $A(X,F)$.
\end{proof}

\begin{theorem}\label{6}
Suppose an intermediate ring $A(X,F)$ is regular. Then $A(X,F)=C(X,F)$.
\end{theorem}

\begin{proof}
Choose $f$ in $C(X,F)$. To show that $f\in A(X,F)$, it suffices to check in view of the absolute convexity of $A(X,F)$ in Lemma \ref{5}, that $\frac{1}{1+\lvert f\rvert}$ is a multiplicative unit of the ring $A(X,F)$. If possible let $\frac{1}{1+\lvert f\rvert}$ be not a unit in $A(X,F)$. Then the principal ideal $(\frac{1}{1+\lvert f\rvert})$ in $A(X,F)$ is a proper one. Since $A(X,F)$ is assumed to be regular and each proper ideal of a regular ring can be easily proved to be a $z^\circ$-ideal, it follows that the principal ideal $(\frac{1}{1+\lvert f\rvert})$ is a $z^\circ$-ideal. But we observe that $Z(\frac{1}{1+\lvert f\rvert})$ is an empty set and therefore $\frac{1}{1+\lvert f\rvert}$ is not a divisor of zero in $A(X,F)$. Since each member of a proper $z^\circ$-ideal is necessarily a divisor of zero, we arrive at contradiction.
\end{proof}

Given a totally ordered field $F$, The following proposition decides the class of spaces $X$ for which $C(X,F)$ is a regular ring.

\begin{theorem}\label{7}
The following statements are equivalent:
\begin{itemize}
\item[1)] $C(X,F)$ is a regular ring.
\item[2)] $X$ is a $P_F$-space.
\item[3)] Every ideal of $C(X,F)$ is a $z^\circ$-ideal.
\end{itemize}
\end{theorem}

\begin{proof}
Equivalence of $(1)$ and $(2)$ is already proved in \cite{AKPG2004}, Theorem 3.2. Since each ideal in a reduced ring is a $z^\circ$-ideal, the implication $(1)\Rightarrow (3)$ is immediate. Assume therefore that $(3)$ is true. This implies clearly that each ideal of $C(X,F)$ is a $z$-ideal, we use the fact that each $z^\circ$-ideal in a semisimple ring is a $z$-ideal. Consequently by Theorem 3.2 in \cite{AKPG2004}, $X$ becomes a $P_F$-space.
\end{proof}

\begin{theorem}\label{8}
Let $X$ be a $P_F$-space. Then for an intermediate ring $A(X,F)$, the following two statements are equivalent:

\begin{itemize}
\item[1)] $A(X,F)=C(X,F)$.
\item[2)] Each ideal of $A(X,F)$ is a $z^\circ$-ideal.
\end{itemize}
\end{theorem}

\begin{proof}
Follows from Theorem \ref{7} and a close look into the proof of Theorem \ref{6}.
\end{proof}

The following new result is a special case of the theorem \ref{8} with the choise $F=\mathbb{R}$.

\begin{corollary}
	With in the class of $P$-spaces $X$, an intermediate ring $A(X)$ containing $C^*(X)$ and contained in $C(X)$ is regular if and only if each ideal of $A(X)$ is a $z^\circ$-ideal.
\end{corollary}	

Let us designate for our convenience the aggregate of all totally ordered fields $F$, which satisfy the following two conditions, by the notation $\mathcal{F}$:

\begin{itemize}
\item[a)] cf(F)$\equiv $ the co-finality index of $F$, is $\omega_\circ$, which means that there is an order isomorphic copy of the ordinal number $\omega_\circ$, which is co-finally embedded in $F$.

\item[b)] $F$ is cauchy complete in the sense that every Cauchy sequence in $F$ is convergent in $F$, with respect to the order topology in $F$.
\end{itemize}

\begin{theorem}\label{9}
Let $F$ be a totally ordered field belonging to the family $\mathcal{F}$. Then for a space $X$, the following two statements are equivalent:

\begin{itemize}
\item[1)] $X$ is an almost $P_F$-space.
\item[2)] $X$ is an almost $P$-space.
\end{itemize}
\end{theorem}

To prove this theorem, we will need the following two subsidiary results which we reproduce to make this article self-contained:

\begin{lemma}\label{10}[Theorem 2.7 \cite{ACG2004}] 
For any topological space $X$ and for any totally ordered field $F$, the zero sets in $X$ with respect to the field $F$ are $G_\delta$ subsets of $X$ if and only if cf(F)$=\omega_\circ$.
\end{lemma}

\begin{lemma}\label{11}\cite{AKPG2004}
Let F be a totally ordered field in $\mathcal{F}$. Furthermore let $V$ be a $G_\delta$ set in a space $X$ with a compact set $S\subseteq V$. Then there exists a zero set $Z$ of $F$-valued continuous function on $X$ such that $S\subseteq Z\subseteq V$.
\end{lemma}

Proof of the theorem \ref{9}: $(2)\Rightarrow (1)$: Let $X$ be almost $P$-space. Let $Z$ be a non-empty zero set of $F$-valued continuous function on $X$. Then by Lemma \ref{10}, $Z$ is a non-empty $G_\delta$ set in $X$ and therefore $int_X Z\neq \phi$. Thus $X$ is almost $P_F$-space. 

$(1)\Rightarrow (2):$ Let $X$ be almost $P_F$-space and $G$ be non-empty $G_\delta$ set in $X$. Choose a point $p$ from $G$. By using Lemma \ref{11}, we can produce a zero set $Z$ in $X$ of $F$-valued continuous function such that $p\in Z\subseteq G$. As $X$ is almost $P_F$-space, it follows that $int_X Z\neq \phi$, consequently $int_X G\neq \phi $. Hence $X$ is almost $P$-space.

The following proposition is a description of almost $P_F$-spaces in terms of an intermediate ring.

\begin{theorem}\label{12}
Let $A(X,F)$ be an intermediate ring. Then $X$ is almost $P_F$ if and only if each fixed maximal ideal $M^p_A\equiv \{f\in A(X,F):f(p)=0\}, p\in X$ of $A(X,F)$ is a $z^\circ$-ideal.
\end{theorem}

\begin{proof}
Let $X$ be almost $P_F$-space and $p\in X$. Choose $f\in M^p_A$ and $g\in P_f\equiv $ the intersection of all minimal prime ideals of $A(X,F)$ which contain $f$. Then by Theorem \ref{3}, $int_X Z(f)\subseteq int_X Z(g)$. Again the hypothesis $X$ is almost $P_F$-space ensures us that $Z(f)= cl_X int_X Z(f)\subseteq cl_X int_X Z(g)=Z(g)$. Since $f\in M^p_A$, we have $f(p)=0$, consequently $g(p)=0$ and therefore $g\in M^p_A$. Then $P_f\subseteq M^p_A$ and hence $M^p_A$ is a $z^\circ$-ideal of $A(X,F)$. To prove the other containment let $X$ be not an almost $P_F$-space. This means that there is an $f\in B(X,F)$ such that $Z(f)\neq \phi $ but $int_X Z(f)=\phi$. Choose a point $p\in Z(f)$. We note that $f$ is not a divisor of zero in $A(X,F)$ and $f\in M^p_A$. Since all the members of a proper $z^\circ$-ideal in a ring are divisors of zero, it follows that $M^p_A$ is not a $z^\circ$-ideal of $A(X,F)$.
\end{proof}

A special case of this theorem with the choice $F=\mathbb{R}$, yields the following new characterization of almost $P$-spaces via an arbitrary intermediate ring $A(X)$ of real valued continuous functions on $X$.

\begin{corollary}
	$X$ is an almost $P$-space if and only if each fixed maximal ideal of $A(X)$ is a $z^\circ$-ideal.
\end{corollary}

We shall now give a characterisation of almost $P_F$-spaces via maximal ideals ( fixed or free ) of the parent ring $C(X,F)$. For this purpose we require a Gelfand-Kolmogoroff like theorem [ see \cite{GJ}, Theorem 7.3] for the description of maximal ideals of $C(X,F)$ and we make an approch similar to that in \cite{GJ}, Chapter 6. We recall the set $\beta_F X$ of all $z_F$-ultrafilters on $X$. For each point $p\in X, A_{p,F} \equiv \{ Z \in Z(X,F):p \in Z \}$ is a (fixed ) $z_F$-ultrafilter on $X$. Therefore $X$ is a ready made index set for the family of all fixed $z_F$-ultrafilters on $X$. Let us enlarge $X$ to a bigger set $\widehat{X}$, to serve as an index set for the family of all $z_F$-ultrafilters on $X$ (fixed and free both). For each $p\in \widehat{X}$, let the corresponding $z_F$-ultrafilter be denoted by $A^p_F$ with the stipulation that whenever $p\in X, A^p_F\equiv A_{p,F}$. For each $Z\in Z(X,F)$, set $\overline{Z}=\{p\in \widehat{X}: Z\in A^p_F\}$. It is not hard to check that $\{ \overline{Z}: Z\in Z(X,F)\}$ is a base for the closed sets of some topology, essentially the Stone-topology on $\widehat{X}$, as mentioned in the introduction of this article. Since there is a one-to-one correspondence between the maximal ideals of $C(X,F)$ and the $z_F$-ultrafilters on $X$ [ see Theorem 2.19(5), \cite{ACG2004}]; via the $Z_F$-map$: M\mapsto Z_F{M}(f\mapsto Z_F(f))$, it is easy to verify that for each $p\in \widehat{X}$, the corresponding maximal ideal $M^p_F$ in $C(X,F)$ is determined by the formula $M^p_F=\{ f\in C(X,F): p\in cl_{\widehat{X}} Z(f)\}$. Here we note that the space $\widehat{X}$, constructed above is just a homeomorphic copy of $\beta_F X$ constructed in \cite{AR2016}, Section 3).

\begin{theorem}\label{13}
$X$ is an almost $P_F$-space if and only if every maximal ideal of $C(X,F)$ is a $z^\circ$-ideal.
\end{theorem}

\begin{proof}
If $X$ is not an almost $P_F$-space, then from Theorem \ref{12}, on choosing $A(X,F)=C(X,F)$, we can say that there exists at least one fixed maximal ideal of $C(X,F)$ which is not a $z^\circ$-ideal. Conversely, let $X$ be almost $P_F$-space. A typical maximal ideal of $C(X,F)$ is of the form $M^p_F$, for some $p\in \widehat{X}$, as described above. Choose $f\in M^p_F$ and $g\in P_f\equiv $ the intersection of all minimal prime ideals in $C(X,F)$ which contain $f$. Then from Theorem \ref{3}, we can write $int_X Z(f)\subseteq int_X Z(g)$, consequently in view of the almost $P_F$ condition on $X$ we see that $Z(f)= cl_X int_X Z(f)\subseteq Z(g)$ and hence $cl_{\widehat{X}} Z(f)\subseteq cl_{\widehat{X}} Z(g)$. Since $f\in M^p_F$, it follows that $p\in cl_{\widehat{X}} Z(f)$, therefore $g\in M^p_F$. Thus $P_f\subseteq M^p_F$ and hence $M^p_F$ is a $z^\circ$-ideal of $C(X,F)$. 
\end{proof}
 
\begin{theorem}\label{14}
Let $X$ be an almost $P_F$-space. Then for an intermediate ring $A(X,F)$, each maximal ideal of $A(X,F)$ is a $z^\circ$-ideal if and only if $A(X,F)=C(X,F)$.
\end{theorem}

\begin{proof}
If $A(X,F)=C(X,F)$, then from Theorem \ref{13}, it is immediatethat each maximal ideal of $A(X,F)$ is a $z^\circ$-ideal. Conversely, let $A(X,F)\neq C(X,F)$. Then there exists an $f\in A(X,F)$, such that $Z(f)=\phi$ but $f$ is not invertible in this ring. Now there exists a maximal ideal $M$ in $A(X,F)$ containing $f$, as $f$ is a nonunit in this ring. Since $f$ is a not a divisor of zero in $A(X,F)$, it follows that $M$ can not be a $z^\circ$-ideal of $A(X,F)$.
\end{proof}

\bibliographystyle{plain}

}
\end{document}